\documentclass[11pt,a4paper]{article}

\usepackage{lipsum}
\usepackage{url}
\usepackage{float}
\usepackage{amsfonts}
\usepackage{amsmath}
\usepackage{amsthm}
\usepackage{graphicx}
\usepackage{epstopdf}
\usepackage{algorithmic}
\usepackage{adjustbox}
\usepackage{tikz}
\usepackage{tikz-3dplot}
\usepackage{verbatim} % for comment
\usepackage{subcaption}
\usepackage{pgfplots}
\usepackage{amssymb}
\usepackage{amsopn}
\usepgfplotslibrary{groupplots}

\ifpdf
  \DeclareGraphicsExtensions{.eps,.pdf,.png,.jpg}
\else
  \DeclareGraphicsExtensions{.eps}
\fi

\newcommand{\iu}{\mathrm{i}} % imaginary unit

\newcommand{\cD}{\mathcal{D}}
\newcommand{\cX}{\mathcal{X}}

\DeclareMathOperator{\Real}{Re}
\DeclareMathOperator{\Imag}{Im}

\theoremstyle{plain}% default
\newtheorem{thm}{Theorem}[section]
\newtheorem{lemma}[thm]{Lemma}
\newtheorem{proposition}[thm]{Proposition}

\theoremstyle{definition}

\newtheorem{conjecture}{Conjecture}[section]

\theoremstyle{remark}
\newtheorem*{remark}{Remark}

\begin{document}
\title{A perfectly matched layer approach \\ for the spectral split-step Pad\'e method}

% Authors: full names plus addresses.
\author{Daniel Walsken$^{1}$,
Matthias Ehrhardt$^{1,}$\footnote{Corresponding Author, email: ehrhardt@uni-wuppertal.de},
Pavel Petrov$^{2}$}

	\date{$^1$ Applied and Computational Mathematics, University of Wuppertal, Gaußstrasse 20, Wuppertal, 42119, Germany\\ 
		$^2$ Instituto de Matem\'atica Pura e Aplicada, Rio de Janeiro, Brazil\\
		E-mail: walsken@uni-wuppertal.de$^1$, 
        ehrhardt@uni-wuppertal.de$^{1,*}$
        pavel.petrov@impa.br$^2$}
        
        \maketitle

    \begin{tikzpicture}[remember picture,overlay]
\node[anchor=north east,inner sep=20pt] at (current page.north east)
	{\includegraphics[scale=0.2]{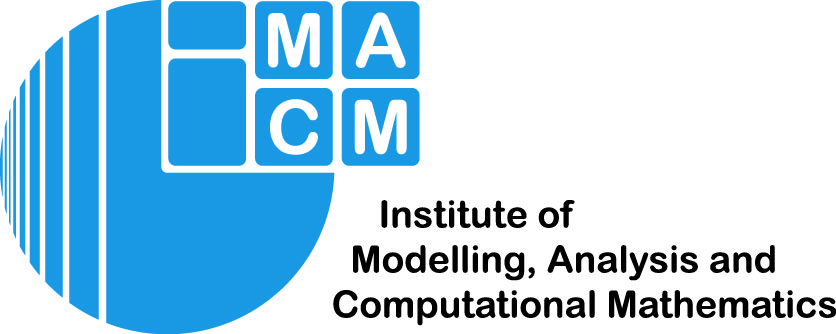}};
\end{tikzpicture}

% REQUIRED
\begin{abstract}
The split-step-Pad\'e (SSP) method is widely used to model wave phenomena in various applications, including radio physics, optics and acoustics. 
In this method, the propagator of the one-way counterpart of the Helmholtz equation is computed through its Pad\'e approximant
and a finite-difference discretization of the transverse operator. 
This work develops and validates numerically a spectral counterpart of the SSP method. 
A key challenge in practical applications is inverting the transverse operator in the presence of perfectly matched layers (PMLs), 
which are commonly used to truncate the computational domain. Such inversion can be accomplished using Krylov subspace methods, 
which converge rapidly, provided that a suitable preconditioner is used. 
We also study the analytical properties of the spectral SSP marching scheme under periodicity conditions in the transverse variable. 
We validate the newly developed spectral SSP method numerically in two realistic test scenarios from radio physics and underwater acoustics.
\end{abstract}



%%%%%%%%%%%%%%%%%%%%%%%%
%\newtheorem{theorem}{Theorem}[section]
%\newtheorem{lemma}[theorem]{Lemma}
%\newtheorem{proposition}[theorem]{Proposition}
%\newtheorem{corollary}[theorem]{Corollary}
%\newtheorem{conjecture}[theorem]{Conjecture}
%\newtheorem{remark}[theorem]{Remark}

%%%%%%%%%%%%%%%%%%%%%%%%%%%%%%%%%%%%%%%%%%%%%%%

%%%%%%%%%%%%%%%%%%%%%%%%%%%%%%%%%%%%%%%%%%%%%%%
\section*{Introduction}
The propagation of waves in inhomogeneous media is a fundamental problem in applied mathematics and computational physics.
It has numerous applications in optics, radio physics, underwater and atmosphere acoustics, and geophysical exploration. 
The governing model is usually the Helmholtz equation, whose direct numerical solution in large or unbounded domains is computationally expensive due highly oscillatory nature of its solutions that imposes severe discretization step restriction.
A common approach to reducing computational complexity is to use parabolic approximations of the Helmholtz equation, which were first introduced by Leontovich and Fock \cite{leontovich1946, vlasov1995parabolic}. 
These models transform the original boundary value problem into an initial value problem in a preferred propagation direction, thereby significantly reducing computational effort. 
Although such reformulation neglects back-scattered waves, it has proven to be extremely efficient in many applications, where parabolic equations have become main computational tools \cite{claerbout1976, jensen2011computational, levy2000parabolic, lu2006some}, especially after the introduction of \textit{split-step Fourier} (SSF) spectral solvers (mostly for narrow-angle parabolic equations). 
 
However, classical parabolic equation methods are limited in their ability to accurately capture wide-angle propagation effects.
This limitation has been overcome by the development of \textit{wide-angle parabolic equations} (WAPEs), which approximate the square-root operator arising from the factorization of the Helmholtz operator using rational Pad\'e approximants \cite{abawi1997coupled, claerbout1976, collins93, jensen2011computational, levy2000parabolic, popov1977}. 
It is important that WAPEs substantially improve angular accuracy while preserving computational efficiency paraxial equations. 
The next major improvement of WAPEs was the invention of the \textit{split-step Pad\'e} (SSP) method \cite{avilov1995, collins93} that aims at approximating the \emph{propagator} of the operator square root (i.e., the exponential that emerges from integration of one-way Helmholtz equation over small step), by a Pad\'e series. 
This has proven to be much more efficient than approximating the square-root operator itself, both in terms of accuracy and computational efficiency. 
A particularly impressive fact about SSP method is that in can be used with integration step of about 10 wavelengths \cite{collins93, PETROV2020115526}.
Classical SSP methods are widely used in many areas such as optics, acoustics and radio physics (both for solving direct and inverse problems) where they are often considered a huge improvement as compared to SSF/WAPE \cite{Lu2003, lu2006some, collins2019parabolic, lin2013, lytaev2018nonlocal, lytaev2018splitirregular, petrov024generalization, wu2023higher, lytaev2026}.

Recently, it was shown that the advantages of the SSF and SSP methods could be combined by computing Pad\'e approximants at each step of the marching scheme via the fast Fourier transform (FFT) \cite{Daniel1}.
Because they are more accurate for smooth solutions, spectral methods are a good choice for solving parabolic wave equations numerically \cite{boyd2001chebyshev, feit1978light}.
In particular, Fourier-based discretizations yield spectral (often exponential) convergence and can be efficiently implemented using FFTs \cite{trefethen2000spectral}.

%Due to their superior accuracy for smooth solutions, spectral methods seem to be a natural choice for the numerical solution of parabolic wave equations, \cite{antoine2020pseudospectral, boyd2001chebyshev, feit1978light, trefethen2000spectral}.
%In particular, Fourier-based discretizations yield spectral (often exponential) convergence and can be efficiently implemented using fast Fourier transforms; see \cite{trefethen2000spectral}. 

A central challenge in the numerical treatment of wave propagation problems is the modeling of unbounded domains \cite{givoli2013numerical, hagstrom1999radiation, shen2009some, TSYNKOV1998465}. 
Truncating the computational domain generally leads to artificial reflections at the boundaries, which can spoil the numerical solution % in the area of interest
due to non-physical reflections at the artificial boundaries. 

Among the various approaches developed to address this issue, \textit{perfectly matched layers} (PML), introduced by B\'erenger \cite{berenger1994perfectly}, have proven to be particularly effective. 
%%% introduced by Bérenger \cite{berenger1994perfectly} and extended to acoustic waves in \cite{liu1997perfectly}.
%
PML involves surrounding the computational domain with an artificial absorbing layer obtained via complex coordinate stretching, which leads to exponential attenuation of outgoing waves without reflection at the artificial boundary.
PML techniques have been extensively studied for various wave propagation models, including the Helmholtz and Maxwell equations
%; see, e.g., 
\cite{chew19943d, hesthaven1998analysis, liu1997perfectly}  %pled2022review}.
Their mathematical properties, such as well-posedness, stability, and convergence, have been analyzed in several works, including \cite{appelo2006perfectly, becache2000analysis, kim2010}. 
However, incorporating PML techniques into split-step Pad\'e methods in combination with spectral discretizations is considerably less developed, although several important ideas have been recently proposed \cite{antoine2020pseudospectral, antoine2022pseudospectral}.

In the spectral SSP framework, introducing a PML leads to additional analytical and numerical challenges. 
In particular, the complex coordinate stretching destroys the simple diagonal structure of the differential operator in Fourier space, 
which is essential for efficient inversion in the classical spectral setting.
Consequently, the resulting operators cannot be inverted analytically in Fourier space,
therefore iterative methods, such as Krylov subspace solvers (e.g., GMRES \cite{saad1986GMRES, trefethen1997numerical}), must be employed. 
This raises important questions regarding stability, efficiency, and the design of suitable preconditioners.

The goal of this work is to develop and analyze a \textit{spectral split-step Pad\'e} %SSP 
(SSSP) method with a PML to efficiently solve mode parabolic equations. 
This approach combines a Fourier spectral discretization in the transverse direction and a Pad\'e approximation of the propagator to achieve wide-angle capabilities. 
It also uses a PML formulation to treat unbounded domains within a finite computational region. 
Particular focus is placed on efficiently realizing the method numerically, including using Krylov subspace methods and constructing effective preconditioners based on operator splitting ideas.

From a mathematical perspective, recent advances have provided a rigorous analytical foundation for these approaches. 
In particular, the square-root operator appearing in parabolic wave equations has been studied within the framework of pseudodifferential operators, and the well-posedness of the associated evolution equations has been established in \cite{ehrhardt2025square}. 
These results justify the use of exponential propagators and their rational approximations within the SSP framework.

We establish the stability and convergence properties of the proposed scheme. 
Specifically, we analyze the dissipativity of the PML-modified operator, prove the stability of the Pad\'e-based marching scheme, and derive a global error estimate combining contributions from the temporal discretization, spectral approximation in space, and PML layer truncation.
These results rigorously justify the method and clarify the interplay between the different approximation components.
We demonstrate the performance of the proposed method on benchmark problems from underwater acoustics, 
including flat-bottom and wedge configurations. 
We also apply the method to radio wave propagation in the atmosphere. 
The numerical results demonstrate that the method achieves high accuracy, even on relatively coarse grids.
This highlights the effectiveness of the spectral discretization and the robustness of the PML formulation.

%%%% Structure of Daniel2 paper
The remainder of the paper is organized as follows. 
Sections~\ref{sec:SSP_classical} and \ref{sec:sssp} introduce the mode parabolic equation and the spectral split-step Pad\'e method in more detail. 
Section~\ref{sec:3} presents the PML formulation and its incorporation into the spectral framework. 
Section~\ref{sec:3} is devoted to the analytical properties of the method, including stability and convergence. 
Section~\ref{sec:radio} describes the application to radiowave propagation in the troposphere,
and Section~\ref{sec:sw_acoust} deals with sound propagation in a 3D shallow-water waveguide.
Finally, Section~\ref{sec:5} provides concluding remarks.

%%%%%%%%%%%%%%%%%%%%%%%%%%%%%%%%%%%%%%%%%%%%%%%%%%%%

%%%%%%%%%%%%%%%%%%%%%%%%%%%%%%%%%%%%
\section{The Split-Step Pad\'e method}\label{sec:SSP_classical}
The \textit{split-step Pad\'e} (SSP) method was proposed to numerically integrate a one-way counterpart of Helmholtz-type elliptic equations. 
For the sake of clarity, we will write this equation as
\begin{equation}\label{HEgen}
    \partial_x^2 U + \partial_y^2 U + k^2U = 0\,,
\end{equation}
where $U = U(x,y)$ is the wavefield, and $k = k(x,y)$ is the medium wave number. 
The physical meaning of $U$ and $k$ depends on the area of application, e.g., it can be transverse component of the electric field in a radiowave propagating in the troposphere or the amplitude of a vertical mode in the normal mode decomposition of the acoustic field (these two examples are considered in Section~\ref{sec:radio} and Section~\ref{sec:sw_acoust}). 

The following one-way counterpart of \eqref{HEgen} can be obtained by a formal factorization of the Helmholtz operator and by retaining only the waves propagating in the positive direction of the $x$ axis:
\begin{equation}\label{eq:square_root}
    \partial_x U (x,y) = \iu\sqrt{\partial_y^2+k^2(x,y)} \,U(x, y) \,,\quad x,y\; \in  [0,R]\times(-\infty,\infty)\,.
\end{equation}
A mathematically rigorous definition of the square root of an operator for the case when its numerical range is bounded away from the real axis from above (this is always the case in real-world problems, as $k$ has a positive imaginary part due to the medium attenuation) is given in \cite{ehrhardt2025square} together with the existence and uniqueness proof for the initial-value problem for \eqref{eq:square_root} in the half-space $x\ge0$.

There are several reasons why one-way approximations Helmholtz-type equations are generally more practical than the latter. 
In particular, and important advantage from the viewpoint of scientific computing consists in the fact that a boundary-value problem for an elliptic equation is transformed into a a Cauchy problem for an evolutionary equation that can be integrated numerically by a marching scheme.

Another advantage is that the rapidly oscillating term $\mathrm{e}^{\iu k_0 x}$ (the principal oscillation) in the propagation direction $x$ defined by the reference wave number $k_0$ can be canceled out by introducing a new unknown function
$u(x,y) = \mathrm{e}^{-\iu k_0x}U(x,y)$. 
Introducing the operator $\mathcal{X} = (\partial_y^2 + k_0^2 - k^2)/k_0^2$, we can write an equation for $u$ as
\begin{equation}
    \partial_x u(x,y) = \iu k_0 \bigl( \sqrt{1 + \cX} - 1 \bigr)u\,.
\end{equation}
The solution of the latter equation can be computed by a marching scheme obtained by formally integrating it along a small interval of length $h$ in $x$ direction, i.e.,
\begin{equation}\label{eq:propagation}
    u(x+h, y) = \exp\Bigl(\iu k_0 h \bigl(\sqrt{1 + \cX } - 1 \bigr)\Bigr) u(x, y) = f(\cX)u(x, y)\,,
\end{equation}
where $f(z) = \mathrm{e}^{\iu t (\sqrt{1+z}-1)}$ is defined as a scalar function of a complex variable, and $t = k_0h$. 
The exponential on the right-hand side of \eqref{eq:propagation} is called \textit{a propagator}.

The \textit{Split-Step Pad\'e (SSP) method} consists in the approximation of the pseudodifferential operator on the right-hand side of \eqref{eq:propagation} by the $[P|P]$ Pad\'e approximant $R_P$ of the form
\begin{equation}\label{eq:pade}
    u(x+h, y) \approx R_P\bigl(\cX\bigr) u(x,y)
    = \biggl( d_0 + \sum_{j=1}^{P} \frac{d_j}{1 + b_j \cX} \biggr) u(x,y) 
    = \biggl( d_0 u + \sum_{j=1}^P d_j w_j \biggr)\,,
\end{equation}
where the auxiliary functions $w_j$ are calculated at each marching step in $x$ by solving the following equations
\begin{equation}\label{wj}
    (1 + b_j \cX) w_j(y)=u(x,y)\,,\quad y\in  (-\infty,\infty),\quad j=1,2,\dots,P.
\end{equation}
Over the years, numerous distinct varieties and extensions of the SSP framework have been proposed in order to adapt to more demanding physical scenarios.
Examples include the high-order pseudo-differential developments by Antoine et al.\ \cite{antoine2020pseudospectral}, 
the multi-dimensional and electromagnetic wide-angle adaptations by Lu \cite{LU1999231}, 
the energy-conserving and reciprocity-preserving one-way wave formulations of Godin \cite{godin1999reciprocity},
and the vector/coupled-mode extensions in curvilinear domains by Petrov et al.\ \cite{petrov024generalization}.

% Split Step spectral (but not Pade):  \cite{clark2009wide}

%\matthias{For an implementation of the Pad\'e approximation (using the Chebyshev–Pad\'e algorithm) we refer to \cite[Section~9.1.]{dedner2001transparent}.}
%%% https://de.mathworks.com/matlabcentral/fileexchange/5234-chebyshev-pade-approximation
%\matthias{or do we use the classical "matrix-method" to compute the coefficients ?}
%\daniel{Currently we use the matrix-method to compute the coefficients. RAM somehow uses 'stability terms' in addition to the Pade order.}
%\matthias{It is well-known that the Pad\'e approximation has numerical convergence issues. This is well described  in \cite[Section~III]{Dronamraju2021}. ME: i can provide here more text. The theory is related to the "Baker-Gammel-Wills conjecture" which is for diagonal, i.e.\ $[P|P]$ Pad\'e approximants.}
%
% MATLAB code \texttt{padeapprox} for robust Pad\'e approximation: this approach using Singular Value Decomposition (SVD) to construct Padé approximants, which eliminates spurious pole-zero pairs (Froissart doublets) and handles nearly singular systems better than standard methods, see \cite{gonnet2013robust}.

%%%%%%%%%%%%%%%%%%%%%%%%%%%%%
\begin{remark}[Computation of Pad\'e coefficients]
\label{rem:pade_coeff}
In this work, the coefficients of the Pad\'e approximant are computed using the classical matrix method. 
This method involves solving the linear system associated with matching the Taylor coefficients of the propagator. 
While this approach is straightforward and widely used, it is well known that computing Pad\'e approximants 
can lead to numerical ill-conditioning, particularly for higher approximation orders. 
Convergence and stability issues are discussed in detail in \cite[Section~III]{Dronamraju2021}.

Several more robust alternatives have been proposed in the literature. 
In particular, Chebyshev-Pad\'e techniques combine rational approximation with Chebyshev expansions to provide improved numerical behavior \cite[Section~9.1.]{dedner2001transparent}.
Furthermore, singular value decomposition (SVD)-based approaches can substantially improve robustness by mitigating the effects of nearly singular systems and eliminating spurious pole-zero pairs (Froissart doublets), see \cite{gonnet2013robust}.

From a theoretical perspective, the convergence behavior of Pad\'e approximants is closely related to fundamental questions in rational approximation theory.
One of the most prominent examples is the Baker-Gammel-Wills conjecture, which concerns the convergence of diagonal Pad\'e approximants to analytic functions. 
Although the conjecture is known to be false in full generality, it has motivated extensive research on the localization of poles and the convergence properties of Pad\'e sequences. 
Since the % split-step Pad\'e 
SSP method relies on comparatively low approximation orders, the classical matrix method was sufficient for all numerical experiments reported in this paper.

In particular, numerical instabilities may manifest as spurious poles, the location of which can significantly impact the stability of the resulting rational propagator, cf.\ Section~\ref{subsec:poles} and in particular Conjecture~\ref{conj1}.
\end{remark}
% The  MATLAB code padeapprox for robust Pad\´e approximation, based on Algorithm 2. The input function can be either a vector of Taylor coefficients or a function handle. 
%The code is freely available as part of Chebfun, see L. N. Trefethen et al., Chebfun software package, https://www.chebfun.org/
% https://github.com/chebfun/chebfun/blob/master/padeapprox.m

The above formulation provides an efficient forward propagation scheme. 
However, in this basic form, the treatment of unbounded domains remains an open issue.
In Subsection~\ref{sec:PML}, we introduce a \textit{perfectly matched layer} (PML) to avoid artificial reflections at the boundaries. 
This work follows the notation of Antoine et al.\  \cite{antoine2020pseudospectral}. 
%A PML for the (finite difference) parabolic wave equation has been published by Levy in \cite{levy2001perfectly} for electromagnetic waves.

Note that the proper definition of the square root $\sqrt{\partial_y^2 + k^2}$, let alone the mathematical properties of \eqref{eq:square_root}, is non-trivial.
A rigorous definition of such operators and existence, uniqueness, and well-posedness for equations of the form \eqref{eq:square_root} were established in a recent paper \cite{ehrhardt2025square}.

%%%%%%%%%%%%%%%%%%%%%%%%%
\subsection{Perfectly Matched Layer}\label{sec:PML}
In wave propagation, a common computational problem is that the wave field on an unbounded, open domain $\Omega_\mathrm{open}$ has to satisfy the Sommerfeld radiation boundary condition \cite{sommerfeld1912greensche}.
% When numerically solving those problems on unbounded domains, those domains cannot be easily truncated without introducing boundaries, which interact with the wave field, producing numerically correct, but physically meaningless solutions.
% Over the years, a number of techniques have been developed to solve this problem.

One of the more recent and computationally simplest methods is the introduction of a \textit{perfectly matched layer} (PML), developed by Ber\'enger \cite{berenger1994perfectly} and applied to a parabolic wave equation by Levy \cite{levy2001perfectly}. 
The eponymous PML is a buffer domain $\Omega_\mathrm{PML}$ around the truncated (physical) domain $\Omega_\mathrm{Phys}$, in which the physical coordinates are stretched and the field is dampened.
In Cartesian coordinates for the simple case of a one-dimensional, symmetric domain of interest, this involves truncating the open domain $\Omega_{\rm open} \supset \Omega_{\rm Phys} = [-L_y^*, L_y^*]$ and introducing a layer $\Omega_{\rm PML} = [-L_y, -L_y^*] \cup [L_y^*, L_y]$ around the physical domain (the domain of interest).
The problem is then numerically solved on the computational domain $\Omega_\mathrm{comp} = \overline{\Omega_\mathrm{phys}} \cup \Omega_\mathrm{PML}$.
To ensure that boundary interactions do not propagate into the physical domain, the coordinate is stretched in the PML
\begin{equation*}
    \tilde{y} = y + \mathrm{e}^{\iu\vartheta_y} \int_{L_y^*}^{y} \sigma(s) \,ds\,, \quad \sigma(s)=0 \quad \text{if} \; s \in \Omega_{\rm Phys}\,.
\end{equation*}
Here $\vartheta_y \in (0, \frac{\pi}{2})$ is a real constant and $\sigma(s)$ is the real-valued absorbing function.
This coordinate transform applied to the operator $\cD = \cD(x,y)$ yields the % new
PML operator
\begin{equation*}
        \cD_\mathrm{PML} = k^2 + \frac{1}{S_y}\,\partial_y\Bigl(  \frac{1}{S_y}\partial_y \Bigr) \,
\end{equation*}
from the chain rule, where
\begin{equation}\label{eq:Sy}
    S_y = S_y(y) =
    \begin{cases}
        1\,, &|y| < L_y^* \, , \\
        1 + \mathrm{e}^{\iu \vartheta_y} \sigma(|y| - L_y)\,, & L_y^* \le |y| < L_y \,.
    \end{cases}
\end{equation}
Note here, that the operator remains unmodified inside the physical domain, is only changed by the introduction of dampening in the PML.
In case of the time step of SSP, one can now write the operator $\cX$ as
\begin{equation}\label{XPML}
    \cX_{\rm PML} = \frac{k^2 - k_0^2 + S_y^{-1}\partial_y (S_y^{-1} \partial_y)}{k_0^2}\,.
\end{equation}
The use of a classical SSP method, i.e., solving the subproblems~\eqref{wj} using finite differences, has already been discussed in the context of radio physics \cite{levy2001perfectly}.
% Extending the wave number $k^2(x,y)$ onto the new computational domain is easily done by constant extrapolation
% \begin{equation}
%     k(x,y) = \begin{cases}
%         k(x,y)\,, & y\in[-L^*, L^*]\,, \\
%         k(x,L^*)\,, & y \in (L^*, L]\,, \\
%         k(x,-L^*)\,, & y \in [-L, -L^*)\,.
%     \end{cases}
% \end{equation}
The PML formulation modifies the differential operator while preserving the structure of the problem as a whole. 
Next, we will discuss how to incorporate this modified operator into the spectral SSP framework.

%%%%%%%%%%%%%%%%%%%%%%%%%%%%%
\section{Spectral SSP Method}\label{sec:sssp}
The main goal of this work is to develop a spectral version of the SSP marching algorithm, 
as described by Equations~\eqref{eq:pade} and \eqref{wj} in Section~\ref{sec:SSP_classical}, which is compatible with PML.
The complication in this case consists in the solution of the equations
\begin{equation}\label{wjPML}
    (1 + b_j \cX_{\rm PML}) w_j(y)= u(x,y)\,, y\in  [-L,L],
\end{equation}
with periodic boundary conditions imposed at $y=\pm L$ (replacing Equation~\eqref{wj} on the infinite interval). 

Observe that the differential part of the operator $\cX_{\rm PML}$ in \eqref{XPML} can be computed in the Fourier domain separately from the algebraic part by the formula
\begin{equation*}
\begin{split}
    S_y^{-1} \partial_y \bigl( S_y^{-1} \partial_y u \bigr)
        % &= \int_{-\infty}^{\infty} S_y^{-1} \partial_y \bigl( S_y^{-1} \partial_y \tilde{u} \bigr)\,d\eta \\
        % &= \int_{-\infty}^{\infty} S_y^{-1} \partial_y \bigl( S_y^{-1} \iu\eta \tilde{u} \bigr)\,d\eta \\
        &= \mathcal{F}^{-1}\bigl\{ |S_y|^{-2} |\eta|^2 \tilde{u}
            - \iu \eta S_y^{-1} \partial_y(S_y^{-1}) \tilde{u} \bigr\}\,,
\end{split}
\end{equation*}
where $\eta$ denotes the dual variable to $y$ and $\tilde{u}=\tilde{u}(\eta) = \mathcal{F}\{u\}$ is the Fourier-transformed function $u$. 
% \pavel{I think we need to reduce previous eq to
% \begin{equation}
%     S_y^{-1} \partial_y \bigl( S_y^{-1} \partial_y u \bigr)
%         = \mathcal{F}^{-1}\bigl\{ |S_y|^{-2} |\eta|^2 \tilde{u}
%             - \iu \eta S_y^{-1} \partial_y(S_y^{-1}) \tilde{u} \bigr\}\,,
% \end{equation}
% }
% \textcolor{red}{This operator is not diagonalizable by a pure sine- or cosine transform, due to the multiplication by $\iu\xi_y$ not being representable in pure sine or cosine space. }
% \pavel{Should we remove this comment?}
%
% \textcolor{red}{Homogeneous Dirichlet boundary conditions therefore have to be dealt with by extending the domain by a mirror domain, as described by Smith and Tappert .}
% \pavel{This one goes to radio physics section}
%
The solution of \eqref{wjPML} therefore becomes a trivial task in the Fourier domain if $\delta k^2(y) = k^2(y) - k_0^2 = 0$ and $S_y \equiv 1$.

Recently it was shown \cite{Daniel1} that for the case of non-homogeneous medium $\delta k^2(y)\neq0$ (albeit without the PML) a spectral counterpart of the SSP method can be constructed using a truncated Neumann series expansion of the functions $w_j(y)$
\begin{equation*}
\begin{aligned}
     w_j &= \sum_{m=0}^{\infty}  \bigl(- \tilde{b}_j(1+b_j\partial_y^2)^{-1}\delta k^2 \bigr)^m (1+ b_j\partial_y^2)^{-1}u \\
    &\approx \sum_{m=0}^{M}  \bigl(- \tilde{b}_j(1 +b_j\partial_y^2)^{-1}\delta k^2 \bigr)^m (1 + b_j\partial_y^2)^{-1}u\,,
\end{aligned}
\end{equation*}
which can be easily computed using a Fourier-domain representation of $(1 + b_j\partial_y^2)^{-1}$. 
However, this technique cannot easily be generalized to the case of a PML since there does not seem to be a computationally efficient way to evaluate the inverse of
 $(1 + b_jS_y^{-1}\partial_y (S_y^{-1} \partial_y))$.

The second possible option is to use a Krylov subspace method to solve \eqref{wj} for the weights of the Pad\'e series.
Consider the PML subproblems
\begin{equation} \label{eq:subproblem_pml}
    \bigl(\mathbb{I} + \tilde{b}_j \delta k^2(y) + \tilde{b}_j \cD \bigr) w_j(y) = u(y)\,,
\end{equation}
in which the differential operator $\cD $ is defined by
\begin{equation*}
    \cD u = \frac{1}{S_y(y)} \partial_y \Bigl(\frac{1}{S_y(y)} \partial_y \Bigr) \,,
\end{equation*}
and $\mathbb{I}$ refers to the identity operator.
On the domain $y\in (-L_y, L_y)$, the differential operator $\partial_y u$ can be evaluated in Fourier space
% \pavel{Here we need some smooth transition from $\eta=\eta_p$ to $\pi p/Ly$! Btw, Daniel, can we perhaps be even more consistent here and use subscript $p$ for indexing Pad\'e terms, while reserving $j$ for FFT discretization?} 
% \daniel{Using $j$ for all spectral decompositions is a great call. We already introduced $y \in (-L_y, L_y)$, which automatically makes our spectral representation a Fourier series. I therefore think this is explicit enough.}
% \begin{equation*}
%     \partial_y u(y) = 
%     \sum_{j=0}^{\infty} \frac{\iu \pi j}{L_y} \,\tilde{u}_j \,\mathrm{e}^{\frac{\iu \pi j}{L_y}y} \,.
% \end{equation*}
\begin{equation*}
    \partial_y u(y) = 
    \sum_{j=-\infty}^{\infty} \iu \eta_j \,\tilde{u}_j \,\mathrm{e}^{\iu \eta_j y} \,,
\end{equation*}
with the discrete Fourier wave numbers $\eta_j = \frac{\pi j}{L_y}$.
Let for this purpose $\tilde{u}_j = \int_{-L_y}^{L_y} u(y) \mathrm{e}^{\frac{-\iu\pi j}{L_y}y} \,dy$ be $j$th Fourier coefficient of $u$. 
%\pavel{please define $\widehat{u}_p$! Also i thing hat is reserved for operators, and for consistency i would use tilde for FT}
The perturbation of the wave number due to medium effects $\delta k^2 = k_0^2 - k^2$ is treated as a multiplication operator.
The $p$th coefficient of the Pad\'e expansion, divided by the squared reference wave number, is referred to as $\tilde{b}_p = b_p/k_0^2$.
Numerically, the Fourier Series is truncated, introducing a spectral discretization
\begin{equation*}
    (\partial_y u)_n = \sum_{j=-N/2}^{N/2-1} \iu \eta_j \,\tilde{u}_j \,\mathrm{e}^{\iu \eta_j y_n} \,.
\end{equation*}
The problem \eqref{eq:subproblem_pml} can then be solved iteratively using a Krylov subspace method with a suitable initial guess. 
Suitable options are GMRES \cite{saad1986GMRES} and Bi-CGSTAB \cite{vanderVorst1992BiCGSTAB}. 
In practice, Bi-CGSTAB is easier to use, because it does not require tuning of a restart period for optimal performance and runs slightly faster.
\section{Analytical Properties of the PML--SSP Method}\label{sec:3}
In this section, we establish the stability and convergence properties of the \textit{spectral split-step Pad\'e} (SSSP) method combined with a PML. 
The analysis is performed in a periodic setting that is consistent with the Fourier discretization employed in the numerical scheme. 
More precisely, we will prove that Pad\'e approximant $R_P(\cX)$ in \eqref{eq:pade} acting on a Sobolev space $H^2_{\mathrm{per}}(-L,L)$ of periodic functions defined on the interval $y \in [-L,L]$ has spectral radius $\rho(R_P(\cX))\le1$. 
On the discrete level this implies the decreasing of the norm at each step of the marching scheme. 

For the sake of simplicity let us assume that the reference wave number is purely real, and that the medium is homogeneous, i.e., $k(x,y) = k_0= \text{const}$. 
Clearly, this simplification does not substantially affect the properties of the \textit{PML-modified operator} $\cX_{\mathrm{PML}}$. 
In our case the latter reduces to  
\begin{equation*}
     \cX_{\mathrm{PML}} = \cD u = \frac{1}{k_0^2}\frac{1}{S_y(y)} \frac{d}{dy}\Bigl( \frac{1}{S_y(y)} \frac{d u}{dy} \Bigr)\,,
\end{equation*}
where the stretching function $S_y(y)$ satisfies
\eqref{eq:Sy} and $\Real S_y(y) \ge1$, $\Imag S_y(y) \ge0$ in $\Omega$.
The remainder of this section is dedicated to the analysis of the properties of the operator $R_P(\cD)$.

%%%%%%%%%%%%%%%%%%%%%%%%%%%%%%
\subsection{The poles of the Pad\'e approximant are located on the lower half-plane}\label{subsec:poles}
In this section we study the properties of the Pad\'e approximant of the function $f(z)$ defined in \eqref{eq:propagation}. 
The principal result that we require for investigation of the stability of the marching scheme is that all the poles of the denominator of its $[P/P]$-Pad\'e approximation are located below the real axis. 

We start with two preparatory results.
%%%% Lemma 3.1. 
\begin{lemma}\label{lem:31} 
The coefficients $c_k$ of the power series expansion $f(z) = \sum_{k=0}^{\infty}c_k z^k$ of the propagator function satisfy the following identity
    \begin{equation}\label{Tay_coeffs}
        \sum_{\substack{k+\ell = m,\\
        k\ge0,\, \ell\ge0}}
        c_k \,\bar{c}_{\ell} = \delta_{0m}\,,
    \end{equation}
where $\delta_{ij}$ denotes the usual Kronecker delta, and the overbar denotes complex conjugation.
\end{lemma}
%%%%%
\begin{proof}
    Let us note that the function $f(z)$ is holomorphic on $\Omega = \mathbb{C}\setminus (-\infty, -1]$,
    and so is the function $\overline{f(\bar{z})}$.
    For $z\in \mathbb{R}\cap \Omega$ we also have 
    \begin{equation*} 
        F(z) \equiv f(z) \overline{f(\bar{z})} - 1 = 0\,. 
    \end{equation*}
    Since the function $F(z)$ defined by the latter equality is also holomorphic on $\Omega$ and identically zero on $\mathbb{R}$, it also vanishes for all $z\in \Omega$, hence $f(z)\overline{f(\bar{z})} \equiv 1$ on this domain.
    
    Let us now expand the functions on the left-hand side of the last equality into power series at $z=0$
    \begin{equation*}
    \biggl(\sum_{k=0}^{\infty} c_k z^k\biggr)
    \overline{\biggl(\sum_{k=0}^{\infty} c_k \bar{z}^{k} \biggr)} = 1\,. 
    \end{equation*}
    Multiplying and combining the terms with identical powers of $z$, we arrive at the formula
    \begin{equation*} 
    \sum_{m=0}^{\infty}\biggl(\sum_{k+\ell =m} c_k \,
    \bar{c}_{\ell} \biggr)z^m = 1\,, 
    \end{equation*}
    from which the condition in the Lemma statement follows.
\end{proof}

Consider now the $[P/P]$-Pad\'e approximant $R_P(z) = T_P(z)/Q_P(z)$ (with $\deg T_P = P$, $\deg Q_P = P$), 
computed from the first $2P+1$ coefficients $c_j$ of the Taylor series of $f(z)$ at $z=0$. 
Note that the coefficients of $T_P(z)$, $Q_P(z)$ are found from the linear system obtained by comparing the terms of the powers $z^0,\dots,z^{2P}$ in the asymptotic equality
$Q(z)f(z)-T(z)=O(z^{2P+1})$. 
Its solution is unique up to a common scalar factor whenever the Hankel determinant
$\det(c_{P+i-j})_{i,j=1}^{P}$ is nonzero, cf.\ \cite{Baker1996}.

%%%%%%%%%%%%%%%%%%%%%%% Lemma
\begin{lemma}[Conjugate symmetry]\label{lem:symm}
The numerator and the denominator of the Pad\'e approximant $R_P(z)$ are related by the following formula
\begin{equation*}
  Q_P(z) = \overline{T_P(\bar{z})}, \qquad z\in \Omega.
\end{equation*}
In particular, this implies $|R_P(z)|=1$ for  $z\in\mathbb{R}\cap \Omega$.
\end{lemma}
\begin{proof}
Consider the polynomials $\tilde{T}(z) := \overline{Q_P(\bar{z})}$ and
$\tilde{Q}(z) := \overline{T_P(\bar{z})}$.
Taking the complex conjugate of
the defining relation $Q_P f - T_P = O(z^{2P+1})$ and 
using the equality $\overline{f(\bar{z})}=1/f(z)$ from the proof of Lemma~\ref{lem:31}, we obtain $\tilde{Q} f - \tilde{T} = O(z^{2P+1})$. 

Since $[P/P]$-Pad\'e coefficients are unique up to a constant scalar factor, it can be concluded that
\begin{equation*}
  \tilde{T} = \lambda T_P  \text{ and } \tilde{Q} = \lambda Q_P
\end{equation*} 
for some $\lambda\in\mathbb{C}$. 
Now for $z\in \mathbb{R}$ we have
\begin{equation*} 
   |R_P(z)|^2 = \Bigl| \frac{T_P(z)}{Q_P(z)} \Bigr|^2 = \frac{T_P(z) \overline{T_P(z)}}{Q_P \overline{Q_P(z)}}= \frac{T_P(z) \overline{T_P(\bar{z})}}{Q_P \overline{Q_P(\bar{z})}} = \frac{T_P \tilde{Q}}{Q_P
   \tilde{P}} = 1\,.
\end{equation*}
Finally, normalizing $Q_P(0)=1$ forces $\lambda=1$. Hence
$Q_P(z)=\overline{T_P(\bar{z})}$. For $y\in\mathbb{R}$ this gives $|R_P(z)| = |T_P(z)|/|Q_P(z)| = 1$.
\end{proof}

%%%%%%%%%%%%%%%%%%%%%%%%%%%%%% Conjecture
\begin{conjecture}\label{conj1}
    For all values of $t$ and $P$, all zeros $z_j$ of $Q_P$ are located below the real axis, i.e., $Q_P(z_j)=0$ implies $\Imag(z_j)<0$.
\end{conjecture}

This conjecture can be easily proven analytically for $P\le4$ by computing the coefficients dependent on $t$ using the matrix method (using symbolic algebra systems).
For higher Pad\'e orders, direct computations for a typical range of values $t\in(0.4,200)$ provide numerical evidence that the roots of $Q_P$ lie below the real axis. 
The results for $P=5$ and $P=9$ are displayed in Figure~\ref{fig:pade_roots}. 
We note that due to the inherent instability of Taylor-to-Pad\'e computations, higher orders $P$ may exhibit roots with slightly positive imaginary part, as described in Remark~\ref{rem:pade_coeff}.

\begin{figure}[htb]
    \centering
    \includegraphics[width=\linewidth]{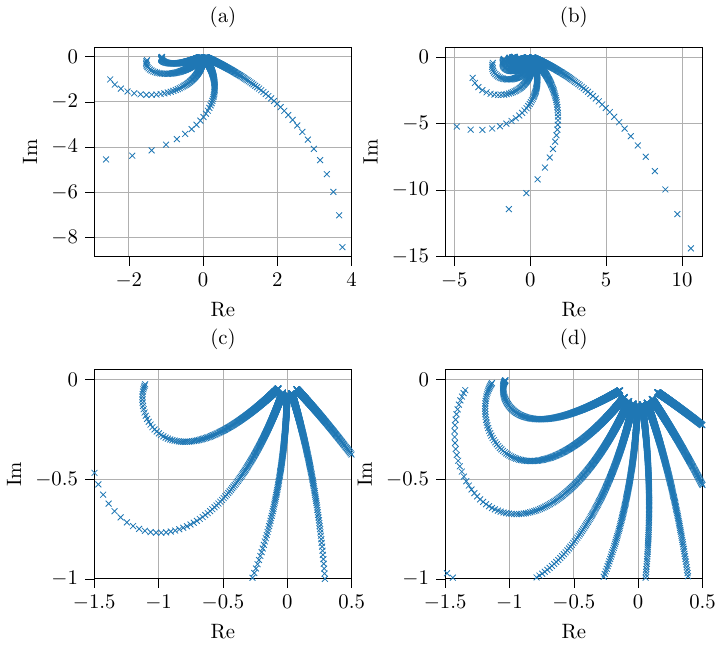}
    \caption{500 numerically computed poles of the Pad\'e expansions of \eqref{eq:propagation}  
    for both $P=5$ ((a) and (c)) and $P=9$ ((b) and (d)) provide justification for Conjecture~\ref{conj1}.
    The poles were computed for $t\in[0.4,0.8, \ldots, 200]$ via the matrix-method.
    The lower row ((c) and (d)) provides a zoomed-in view for clarity.
        % Depending on the stability of the computation of the Pad\'e coefficients, roots with positive imaginary parts may appear for large $P$.
        }
    \label{fig:pade_roots}
\end{figure}

%%%%%%%%%%%%%%%%%%%%%%%%%%%%%% Proposition
\begin{proposition}\label{Prof_RP_ineq}
From Conjecture~\ref{conj1} it follows that $|R_P(z)|\le1$ for all $z$ such that $\Imag(z)\ge0$ with the equality achieved exclusively for $z\in \mathbb{R}$. 
\end{proposition}
%%%%
\begin{proof}
    %\pavel{May be it's better moving first part to the previous Lemma.} \daniel{I think this way is more understandable to read. better than giving the first part earlier to come back to it now.}
    Assume that $\{ z_p\}$, $p = 1,\dots, P$ are zeros of $Q_P$. 
    It follows from Lemma~\ref{lem:symm} that $\{\bar{z}_p\}$ are zeros of $T_P$. 
    Assuming the Conjecture~\ref{conj1}, we have $\Imag(z_p) < 0$ and, accordingly, $\Imag(\bar{z}_p) > 0$.  

    Let us now rewrite $R_P$ in the following way
    \begin{equation}\label{RP_fact}
       R_P = \varkappa \prod_{p=1}^P \frac{z-\bar{z}_p}{z - z_p}\,, 
    \end{equation}
    where $|\varkappa| = 1$ since $|R_P(z)|=1$, $z\in \mathbb{R}$. 
    Let us explicitly introduce real and imaginary parts of the numbers $z=\xi+\iu\eta$ and $z_p=\xi_p+\iu\eta_p$.
    Then the magnitudes of the factors in \eqref{RP_fact} can be estimated as
    \begin{equation*} 
        |z-\bar{z}_p|^2 = (\xi- \xi_p)^2 + (\eta + \eta_p)^2\,, 
    \end{equation*}
    \begin{equation*} 
    |z- z_p|^2 = (\xi- \xi_p)^2 + (\eta - \eta_p)^2\,. 
    \end{equation*}
    Assuming now $\eta \ge 0$ and recalling that $\eta_j<0$ we observe that
    %\begin{equation*} 
    $|\eta - \eta_p|^2 \ge |\eta + \eta_p|^2\,$, 
    %\end{equation*}
    which implies 
    %\begin{equation*}
    $|z  - \bar{z}_p| \le |z - z_p|\,$. 
    %\end{equation*}
    In other words, each factor in \eqref{RP_fact} satisfies
    $\bigl|\frac{z - \bar{z}_p}{z - z_p} \bigr|\le1,$
    and the equality is fulfilled if and only if $\eta =0$. 
\end{proof}

%%%%%%%%%%%%%%%%%%%%%%%
\subsection{Sesquilinear Form and the Spectrum of $\cD$}
In the previous subsection we showed that the poles of the symbol of the propagator's Pad\'e approximation \eqref{eq:pade} are located in the lower half-plane. 
If we could prove that the spectrum (or, say, numerical range) of the operator $\cX$ is restricted to the upper half-plane, the inequality from the Proposition~\ref{Prof_RP_ineq} would guarantee non-increasing of the norm by the SSP marching scheme \eqref{eq:pade}. 
In order to achieve this objective, we follow roughly the same path as the authors of \cite{kim2010} in our periodic setting.

Let us consider the following non-Hermitian bounded sesquilinear form
\begin{equation}\label{sesquiform}
    \langle u,v\rangle_S = \int_{-L}^L S(y) u(y) \bar{v}(y) \,dy
\end{equation}
on $H^2_{\mathrm{per}}(-L,L)$ that satisfies the property $\langle u,v\rangle_S = \overline{\langle u,v\rangle}_{\bar{S}}$.

%%%%%%%%%%%%%%%%%%%%%%%%%%%%%%%%%%%%%%%
\begin{proposition}\label{prop:eigenval}
    Every eigenvalue $\lambda$ of the operator $\cD$ defined on $H^2_{\mathrm{per}}(-L,L)$ has the positive imaginary part, i.e.\ $\Imag(\lambda)\ge 0$.
\end{proposition}
\begin{proof}
    Assume that $\lambda$ is an eigenvalue of $\cD$, and $u_{\lambda}(y)$ its %respective 
    eigenfunction, i.e.,
    \begin{equation*} 
    \cD u_{\lambda} = \lambda u_{\lambda}\,, \quad u_{\lambda} \neq 0\,.
    \end{equation*}
    We substitute both sides of this equality into $\langle\cdot, u_{\lambda} \rangle_S$. 
    On the left-hand side we have
    \begin{multline*}
    \langle \cD u_{\lambda},u_{\lambda}\rangle_S 
    = \int_{-L}^L S(y)\,\cD u_{\lambda}(y) \bar{u}_{\lambda} \,dy 
    = \frac{1}{k_0^2} \int_{-L}^L \partial_y \Bigl(\frac{1}{S} \partial_y u_{\lambda}\Bigr) \bar{u}_{\lambda} \,dy \\ 
    =\frac{1}{k_0^2}\biggl[\frac{1}{S}\partial_y u_{\lambda} \bar{u}_{\lambda} |_{-L}^L - \int_{-L}^L \frac{1}{S}\partial_y u_{\lambda} \partial_y \bar{u}_{\lambda} \,dy \biggr] 
    = -\frac{1}{k_0^2}\int_{-L}^L \frac{1}{S} |\partial_y u_{\lambda}|^2 \,dy 
    \overset{\text{by def}}{\equiv} \alpha + \iu\beta\,.     
    \end{multline*}
    Observing that $\frac{1}{S} = \frac{1 - \iu\sigma}{1+\sigma^2}$ (where $\sigma(y)$ is real and non-negative), we conclude that 
    \begin{equation*} 
    \alpha =  -\frac{1}{k_0^2}\int_{-L}^L \frac{|\partial_y u_{\lambda}|^2}{1+\sigma^2}  dy \le 0\,, \qquad 
    \beta =  \frac{1}{k_0^2}\int_{-L}^L \frac{\sigma |\partial_y u_{\lambda}|^2}{1+\sigma^2}  dy \ge 0\,. 
    \end{equation*}

    At the same time, on the right-hand side we obtain
    \begin{equation*} 
    \langle \lambda u_{\lambda}, u_{\lambda} \rangle_S = \lambda \int_{-L}^L S |u|^2 dy \overset{\text{by def}}{\equiv} \lambda (\mu + \iu\nu)\,,
    \end{equation*}
    where $\mu >0$, $\nu >0$. Equality $\langle \cD u_{\lambda}, u_{\lambda} \rangle_S = \langle \lambda u_{\lambda}, u_{\lambda} \rangle_S$ therefore implies that
    \begin{equation*} 
    \lambda = \frac{\alpha + \iu \beta}{\mu + \iu\nu } = \frac{(\alpha + \iu \beta)(\mu - \iu\nu)}{\mu^2 + \nu^2 }\,.
    \end{equation*}
    Since both $\mu\beta$ and $-\nu\alpha$ are non-negative, it is clear that $\Imag(\lambda)\ge 0$.
\end{proof}

%%%%%%%%%%%%%%%%%%%%%%%%%
\subsection{Spectral Radius of $R_P(\cD)$}
To conclude this section, we emphasize that our iterations given by \eqref{eq:pade} transform the functions from $H^2_{\mathrm{per}}(-L,L)$ to the functions from the same space. 
This fact follows from regularity theorems \cite{evans2022} that are usually formulated for the case of Dirichlet boundary conditions but can be also proven in the periodic case without any substantial reformulation.

%%%%%%%%%%%%%%%%%%%%%%%%%%%%%%%%%%%%%%%%%%%%%%
\subsection{Discrete Fourier Stability}
We now investigate the stability of the fully discrete scheme from the perspective of its Fourier representation.
Because the spatial discretization is based on a Fourier spectral method, analyzing the behavior of individual Fourier modes is a natural approach.
To this end, we express the numerical solution at a fixed range step in terms of its discrete Fourier series,
\begin{equation*}
     u_N(y) = \sum_{j=-N/2}^{N/2-1} \hat{u}_j\, \mathrm{e}^{\iu \eta_j y}, 
\qquad \eta_j = \frac{\pi j}{L_y}.
\end{equation*}
In this representation, differentiation with respect to $y$ corresponds to a multiplication by $\iu\eta_j$, 
allowing us to interpret the action of the differential operators mode by mode.

Without a PML, the second derivative operator acts diagonally in Fourier space, with eigenvalues $-\kappa_j^2$. 
The presence of the PML modifies this structure via the complex stretching function $S_y(y)$. 
This formally leads to a modified symbol of the form
%\begin{equation*}
   %-\frac{\kappa_p^2}{S_y^2(y)},
%\end{equation*}
$-\kappa_j^2/S_y^2(y)$,
which introduces a complex-valued damping factor. 
Since $\Real S_y(y) \ge 1$ and $\Imag S_y(y) \ge 0$, the real part of this expression is non-positive, 
indicating that the PML induces attenuation of the Fourier modes, particularly for large wave numbers.

Combining this with the bounded contribution of the wave number term $k^2(y)$, 
we obtain an effective symbol for the operator $\cX_{\mathrm{PML}}$ of the form
\begin{equation*}
     \lambda_j(y) = \frac{k^2(y) - k_0^2 - \eta_j^2 / S_y^2(y)}{k_0^2}.
\end{equation*}
The behavior of the numerical scheme can thus be understood by examining how each mode propagates under the Pad\'e approximation.

For a single propagation step, each Fourier mode is multiplied by an amplification factor of the form
\begin{equation*}
    G_j = \mathrm{e}^{\iu k_0 h} \biggl(d_0 + \sum_{j=1}^P \frac{d_j}{1 + b_j \lambda_p} \biggr).
\end{equation*}
The exponential prefactor has unit modulus and therefore does not affect stability. 
The remaining expression depends on the location of $\lambda_p$ in the complex plane.

As discussed earlier, the real part of $\lambda_p$ is bounded due to the dissipative contribution of the PML. 
Specifically, the damping term $-\eta_j^2 / S_y^2(y)$ shifts high-frequency modes further into the stable region of the complex plane. 
Since the Pad\'e coefficients satisfy $\Real b_p > 0$, the denominators $1 + b_p \lambda_p$ remain bounded away from zero,
and the amplification factor remains finite.

This implies that the magnitude of $G_p$ does not grow uncontrollably with $p$. 
More precisely, there exists a constant $C$ such that
%\begin{equation*}
   $|G_p| \le 1 + C h$,
%\end{equation*}
uniformly in $p$ and $N$, for sufficiently small step size $h$. 
Thus, the discrete solution satisfies the stability bound
\begin{equation*}
    \|u^{n+1}_N\|_{\ell^2} \le (1 + C h) \,\|u^n_N\|_{\ell^2},
\end{equation*}
where the norm is defined via the Fourier coefficients.
% re\matthias{Do we have to define the discrete L2 norm?}

We observe that the presence of the PML stabilizes the scheme by introducing additional damping for high-frequency components. 
This prevents the accumulation of unresolved oscillations, thereby contributing to the robustness of the method in practical computations.

\section{Radiowave Propagation in the Troposphere}\label{sec:radio}
Wide-angle parabolic \linebreak equations based on the Pad\'e expansion of an operator square root exponential appear in the literature in several different contexts. 
In the first example, we address the problem of tropospheric radiowave propagation. 
In this case, equation is solved on a vertical half-plane ($x,z$ coordinates, $z$ is the altitude, $z=0$ is the Earth's surface). 

\subsection{The Problem Statement}
In two-dimensional Cartesian coordinates, the horizontally polarized, time-harmonic (monochromatic) electric field component \linebreak $E_y(x, z) =  \psi(x, z)$ can be described by the 2D %two-dimensional 
 Helmholtz equation
\begin{equation}\label{HE_rad}
    \partial_x^2 \psi + \partial_z^2 \psi + k_0^2m^2\psi = \delta(x)Q(z)\,,
\end{equation}
where $k_0$ is the wave number of the vacuum, and $m^2(x, z) = 1 + 2n(x, z) + 2z/R$. 
Here, $n(x, z)$ is the variable refractive index, $R$ is the curvature radius of the Earth's surface, and $Q(z)$ models the geometry of the antenna. 
This problem setup is very common across the radio physics community \cite{levy2000parabolic, lytaev2018splitirregular, lytaev2018nonlocal, lytaev2026}.

Assuming that back-scattering in the horizontal direction is negligible, one can replace the BVP for \eqref{HE_rad} by a Cauchy problem for its one-way counterpart
\begin{equation*}
    \partial_x u = \iu k \bigl( \sqrt{1 + \cX} - 1 \bigr)u\,, \quad u|_{x=0} = u_0(z)\,,
\end{equation*}
where $u(x,z) = \mathrm{e}^{-\iu k_0 x}\psi(x,z)$ is an envelope function of a rapidly oscillating electric field, and the Cauchy data $u_0(z)$ is defined through $Q(z)$.

The ground can be represented by a homogeneous Dirichlet boundary condition. 
Due to changes in the refractive index,
the radio field is usually partially reflected at the boundary between atmospheric layers. 
The rest of the field then escapes into the higher layers of the atmosphere. 
Efficient simulation requires one-sided truncation of the domain, which was achieved using a PML. 
In this scenario, homogeneous Dirichlet boundary conditions can be addressed by extending the domain with a mirror domain,
as described by Smith and Tappert \cite{smith1993umpe}.

%%%%%%%%%%%% Section 4.2.
\subsection{Numerical Results}

%%%%%%%%%%%%%%%%%%%%%%%%%%%%%%%%
% \subsection{Radio Physics}
\begin{figure}[H]
    \centering
    \includegraphics[width=\linewidth]{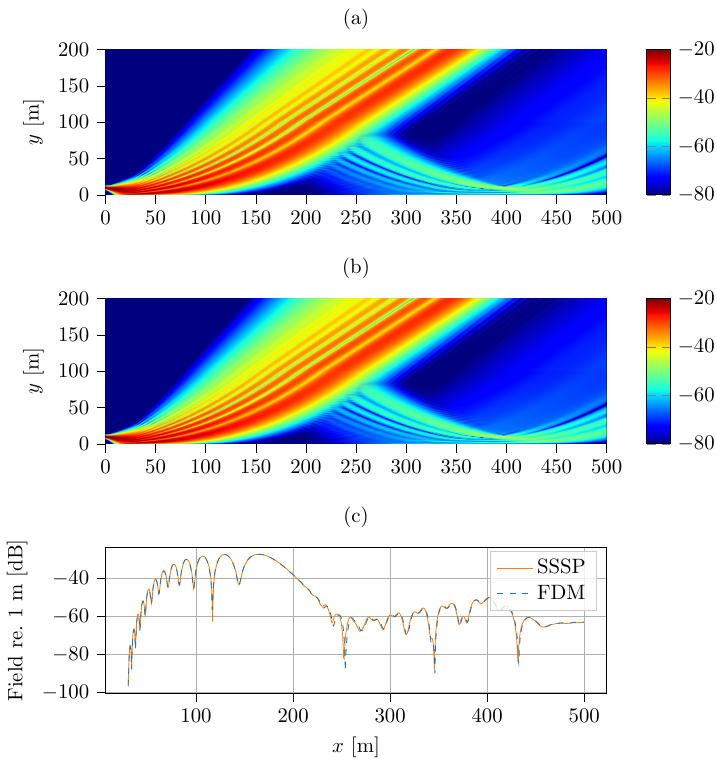}
%     \caption{Propagation of a radio wave with $\lambda=1$\,m through a discontinuous sound speed profile. 
% The discontinuity at $z=100$\,m produces reflections in the field. 
% Both methods were run with the $[1/1]$-Pad\'e expansion.
% All fields are displayed in [dB], relative to the field 1\,m from the source.
% The first plot shows a reference field computed by a classical finite difference (FD) based SSP method. The second plot is produced by the proposed SSSP method. The third plot shows the field in dB, measured at a height of $30$\,m.}
    \caption{Radio wave propagation. (a) shows a reference field computed by Finite Difference SSP, (b) shows the field computed by SSSP and (c) shows a comparison along a slice at $z=30$\,m. All fields are displayed in dB~re.~1\,m.}
    \label{fig:radio}
\end{figure}
The computation is an example of a radio wave with a frequency of 300\,MHz emitted from a parabolic antenna and propagating through the lower atmosphere.
We use the initial condition is modeled as
\begin{equation*}
    u_0(z) = \frac{k_0 \beta}{2 \pi \log_{10}(2)} 
    \,\mathrm{e}^{-\iu k_0 \theta_0 z} 
    \,\mathrm{e}^{-\frac{\beta^2 k_0^2}{8 \log_{10}(2)} (z - z_0)},
\end{equation*}
with $\beta=10^\circ$, $\theta_0 = 10^\circ$ and $k_0=\frac{\omega}{c_0} = 2\pi$\,m$^{-1}$.
 The modified refractive index is defined via
 \begin{equation*}
     m^2(z) = \begin{cases}
        z/250 -1/2 \quad &\textrm{if} ~ z\in[0,100] \\
        0 \quad &\textrm{else}
    \end{cases}.
\end{equation*}
This setting was chosen to show the method's capability to reproduce reflections caused by discontinuities in the equations coefficients representing media interfaces.
A PML with a width of 10 meters was used above $z=200$\,m.
Both the spectral as well as the finite difference method were run for $P=1$ with the step size of 0.1\,m in $x$ as well as 2048 points in $z$.
For the SSSP simulation, BiCGStab was employed to solve the system, using a multiplicative preconditioner 
$\cX_0^{-1} = (1 + k_0^2n^2)^{-1}(1 + \partial_y^2)^{-1}$ and a relative tolerance of $10^{-6}$ as stopping criterion.

The radio signal gets focused into the upper atmosphere by the gradient in the refractive index in the lower troposphere.
The discontinuity at $z=100$\,m introduces reflections of the field back towards the Earth's surface $z=0$.

As shown in Fig.~\ref{fig:radio}, the SSSP solution is in excellent agreement well with the finite difference SSP, even at comparably small number of points in $z$.

%%%%%%%%%%%%%%%%%
\section{Sound Propagation in a 3D Shallow-Water Wave\-guide}\label{sec:sw_acoust}
%\section{Sound Propagation in a Three-Dimensional Shallow-Water Waveguide}\label{sec:sw_acoust}
The second case is related to the propagation of acoustic waves in a three-dimensional shallow water environment. 
This problem can be reduced to solving several one-way counterparts of the 2D Helmholtz equation within the framework of adiabatic normal modes \cite{jensen2011computational, PETROV2020115526}. 
A domain truncation issue naturally arises because the area of interest is typically a small part of the ocean, and reflection-free transmission through its boundaries is essential in practical problems.

%%%%%%%%%%%%%%%%%%%%%%%%%%%%
\subsection{Mode Parabolic Equations}\label{sec:2MPE}

We consider the stationary acoustic field $p=p(x,y,z)$ produced by a time-harmonic point source of the frequency $f$ is located at $x=0, y=0, z=z_s$, that can be described by the \textit{three-dimensional Helmholtz equation} \cite{jensen2011computational}
\begin{equation}
    \begin{split}\label{eq:3DHelmholtz}
       &\partial_x^2 p + \partial_y^2 p + \partial_z^2 p + \frac{\omega^2}{c^2}  \,p
        = -\delta(x, y, z-z_s), \quad \mathrm{in}\; \Omega\,,\\
        & p|_{z=H} = 0\,,\quad p|_{z=0} = 0\,,
        % Sommerfeld condition
        %\lim_{r\to\infty}\left(\partial_n - i\frac{\omega}{c(r\sin\phi,r\cos\phi,z)}\right)p(r\sin\phi,r\cos\phi,z) &= 0
    \end{split}
\end{equation}
in the domain $\Omega = (-\infty,\infty) \times (-\infty, \infty) \times [0,H]$, 
where $z=0$ is the sea surface, $z=H$ is lower boundary of computational domain
(typically taken to be sufficiently large in order to incorporate one or several layers of sediments). 
It is also assumed that a suitable Sveshnikov-type radiation boundary condition is imposed for $r= \sqrt{x^2+y^2} \to \infty$ in order to guarantee uniqueness of the solution of Eq.~\eqref{eq:3DHelmholtz} (see, e.g., \cite{liu2013uniqueness}).
Here, $\omega = 2\pi f$ is the angular frequency, and $c(x,y,z)$ denotes the speed of sound.

One then employs the \textit{modal expansion ansatz}
\begin{equation}\label{eq:modeansatz}
    p(x,y,z) = \sum_{j=1}^{N_m} \mathcal{A}_j(x,y) \,\phi_j(z;x,y)\,,
\end{equation}
where $\phi_j(z;x,y)$ are mode functions computed from the media parameters at the point $(x,y)$, and $\mathcal{A}_j(x,y)$ are the corresponding mode amplitudes.

Under adiabatic assumption it can be shown \cite{jensen2011computational} mode amplitudes satisfy a two-dimensional Helmholtz equation
\begin{equation}\label{eq:modeamplitude}
   \partial^2_x \mathcal{A}_j + \partial^2_y \mathcal{A}_j + k_j^2  \mathcal{A}_j = -\delta(x) \delta(y)\frac{\phi_j(z_s)}{\rho_w}\,.
\end{equation}
The one-way counterpart of Eq.~\eqref{eq:modeamplitude} is called pseudodifferential \textit{mode parabolic equation} (MPE) 
\begin{equation}\label{eq:ampl_evolution}
    \partial_x \mathcal{A}_j = \iu \sqrt{\partial_y^2 + k^2_j}\,\mathcal{A}_j\,, \quad x,y \in (0,x_{\max})\times(-L_y^*,L_y^*)\,,
\end{equation}
which identical to Eq.~\eqref{eq:square_root}.

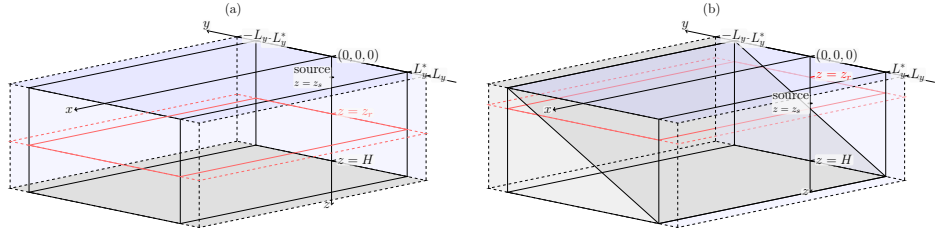
\begin{figure}[H]
    \centering
    %%% scale the tikz picture
    % \begin{adjustbox}{width=0.9\textwidth}
    % \input{images/tikz_plots/ocean_homog}
    %  \end{adjustbox}
     \begin{subfigure}{0.49\linewidth}
        \begin{adjustbox}{width=0.99\linewidth}
         \tikzset{
  labelbox/.style={
    fill=gray!5,        % background color
    fill opacity=0.8,  % how transparent the background is
    text opacity=1,    % text stays fully opaque
    inner sep=1pt,     % padding around the text
    rounded corners=2pt
  }
}

% viewing angles: adjust if you want a different perspective
\tdplotsetmaincoords{78}{135}

\begin{tikzpicture}[scale=2,tdplot_main_coords]
  %--------------------------------------------------
  % Parameters (geometric)
  %--------------------------------------------------
  \def\Lx{6}   % basin length in x-direction
  \def\Lpml{0.5} % size of PML layers
  \def\Ly{2}   % basin width in y-direction
  \def\H{2}  % basin depth (z from 0 at surface to H at bottom)
  \def\zs{0.4} % source depth
  \def\zr{1.1} % receiver plane depth
  \def\ticklen{0.12} % tick length

  % We map physical depth z>0 to negative tikz z-coordinate: (x,y,-z)
  % so that the z-axis visually points downward from the surface.

  %--------------------------------------------------
  % Basin geometry (rectangular tank)
  %--------------------------------------------------

  %--------------------------------------------------
  % Bottom plane
  %--------------------------------------------------
  \fill[black!20,opacity=0.6]
    (0,-\Ly-\Lpml,-\H) -- (\Lx,-\Ly-\Lpml,-\H) -- (\Lx,\Ly+\Lpml,-\H) -- (0,\Ly+\Lpml,-\H) -- cycle;

  %--------------------------------------------------
  % Water fill (semi-transparent)
  %--------------------------------------------------
  % top surface
  \fill[blue!20,opacity=0.4]
    (0,-\Ly-\Lpml,0) -- (\Lx,-\Ly-\Lpml,0) -- (\Lx,\Ly+\Lpml,0) -- (0,\Ly+\Lpml,0) -- cycle;
  % front side (y=0)
  \fill[blue!10,opacity=0.4]
    (0,-\Ly-\Lpml,0) -- (\Lx,-\Ly-\Lpml,0) -- (\Lx,-\Ly-\Lpml,-\H) -- (0,-\Ly-\Lpml,-\H) -- cycle;
  % left side (x=0)
  \fill[blue!10,opacity=0.4]
    (0,-\Ly-\Lpml,0) -- (0,\Ly+\Lpml,0) -- (0,\Ly+\Lpml,-\H) -- (0,-\Ly-\Lpml,-\H) -- cycle;

  % top rectangle + PML (water surface) at z = 0
  \draw[dashed]
    (0,-\Ly-\Lpml,0) -- (\Lx,-\Ly-\Lpml,0) -- (\Lx,\Ly+\Lpml,0) -- (0,\Ly+\Lpml,0) -- cycle;

  % bottom rectangle + PML at z = H  (drawn as z = -H)
  \draw[dashed]
    (0,-\Ly-\Lpml,-\H) -- (\Lx,-\Ly-\Lpml,-\H) -- (\Lx,\Ly+\Lpml,-\H) -- (0,\Ly+\Lpml,-\H) -- cycle;

  % top rectangle (water surface) at z = 0
  \draw[thick]
    (0,-\Ly,0) -- (\Lx,-\Ly,0) -- (\Lx,\Ly,0) -- (0,\Ly,0) -- cycle;

  % bottom rectangle at z = H  (drawn as z = -H)
  \draw[thick]
    (0,-\Ly,-\H) -- (\Lx,-\Ly,-\H) -- (\Lx,\Ly,-\H) -- (0,\Ly,-\H) -- cycle;

  % vertical edges
  \draw[thick] (0,-\Ly,0) -- (0, -\Ly, -\H);
  \draw[thick] (\Lx,-\Ly,0) -- (\Lx, -\Ly, -\H);
  \draw[thick] (0,\Ly,0) -- (0, \Ly, -\H);
  \draw[thick] (\Lx,\Ly,0) -- (\Lx, \Ly, -\H);

  % vertical edges+PML
  \draw[dashed] (0,-\Ly-\Lpml,0)      -- (0,-\Ly-\Lpml,-\H);
  \draw[dashed] (\Lx,-\Ly-\Lpml,0)    -- (\Lx,-\Ly-\Lpml,-\H);
  \draw[dashed] (\Lx,\Ly+\Lpml,0)  -- (\Lx,\Ly+\Lpml,-\H);
  \draw[dashed] (0,\Ly+\Lpml,0)    -- (0,\Ly+\Lpml,-\H);

  %--------------------------------------------------
  % Coordinate axes (origin at back-left, surface corner)
  %--------------------------------------------------
  \draw[->,thick] (0,0,0) -- (\Lx+0.8,0,0) node[labelbox,anchor=east] {\Large$x$};
  \draw[->,thick] (0,\Ly+\Lpml+0.8,0) -- (0,-\Ly-\Lpml-0.8,0) node[labelbox,anchor=south] {\Large$y$};
  % z-axis pointing downward (depth)
  \draw[->,thick] (0,0,0) -- (0,0,-\H-0.85) node[labelbox,anchor=east] {\Large$z$};

  % Propagation direction
  % \draw[->,thick] (\Lx/4,-\Ly/4*3,0) -- (\Lx/2, -\Ly/4*3,0) node[labelbox,anchor=north,align=left] {propagation\\direction};

  %--------------------------------------------------
  % Receiver plane at z = z_r  (horizontal)
  %--------------------------------------------------
  \draw[red!60,thick,dashed]
    (0,-\Ly-\Lpml,-\zr) -- (\Lx,-\Ly-\Lpml,-\zr) -- (\Lx,\Ly+\Lpml,-\zr) -- (0,\Ly+\Lpml,-\zr) -- cycle;
  \draw[red!60,thick]
    (0,-\Ly,-\zr) -- (\Lx,-\Ly,-\zr) -- (\Lx,\Ly,-\zr) -- (0,\Ly,-\zr) -- cycle;

  %---------------------------------------------------
  % Tick labels
  %---------------------------------------------------
  \draw[thick] (0,-\Ly,0) -- ++(-\ticklen,0,0) node[labelbox,anchor=west] { \Large $-L_y^*$ };
  \draw[thick] (0,\Ly,0) -- ++(-\ticklen,0,0) node[labelbox,anchor=west] {\Large $L_y^*$ };
  \draw[thick] (0,-\Ly-\Lpml,0) -- ++(-\ticklen,0,0) node[labelbox,anchor=west] {\Large $-L_y$ };
  \draw[thick] (0,\Ly+\Lpml,0) -- ++(-\ticklen,0,0) node[labelbox,anchor=west] { \Large$L_y$ };

  \draw[thick] (0,0,0) -- ++(-\ticklen,0,0) node[labelbox,anchor=west] {\Large $(0,0,0)$ };

  \draw[thick] (0,0,-\H) -- ++(-\ticklen,0,0) node[labelbox,anchor=west] {\Large $z=H$ };

  \draw[red!60,thick] (0,0,-\zr) -- ++(-\ticklen,0,0) node[labelbox,anchor=west] { \Large$z=z_r$};

  %--------------------------------------------------
  % Acoustic point source at (x,y,z) = (0,0,z_s)
  %--------------------------------------------------
  \draw[fill=black] (0,0,-\zs) circle (0.04);
  \node[labelbox,anchor=east,align=left] at (0,0,-\zs)
    {\Large source\\$z=z_s$};

  \node[above,font=\Large] at (current bounding box.north) {(a)};

\end{tikzpicture}
        \end{adjustbox}
     \end{subfigure}
     \begin{subfigure}{0.49\linewidth}
        \begin{adjustbox}{width=0.99\linewidth}
         \tikzset{
  labelbox/.style={
    fill=gray!5,        % background color
    fill opacity=0.8,  % how transparent the background is
    text opacity=1,    % text stays fully opaque
    inner sep=1pt,     % padding around the text
    rounded corners=2pt
  }
}

% viewing angles: adjust if you want a different perspective
\tdplotsetmaincoords{78}{135}

\begin{tikzpicture}[scale=2,tdplot_main_coords]

  %--------------------------------------------------
  % Parameters (geometric)
  %--------------------------------------------------
  \def\Lx{6}   % basin length in x-direction
  \def\Lpml{0.5} % size of PML layers
  \def\Ly{2}   % basin width in y-direction
  \def\H{2}  % basin depth (z from 0 at surface to H at bottom)
  \def\zs{0.9} % source depth
  \def\zr{0.4} % receiver plane depth
  \def\ticklen{0.12}

  % We map physical depth z>0 to negative tikz z-coordinate: (x,y,-z)
  % so that the z-axis visually points downward from the surface.

  %--------------------------------------------------
  % Basin geometry (rectangular tank)
  %--------------------------------------------------

  %--------------------------------------------------
  % Bottom (wedge)
  %--------------------------------------------------
  \fill[black!20,opacity=0.6]
    (0,-\Ly,0) -- (\Lx,-\Ly,0) -- (\Lx,\Ly,-\H) -- (0,\Ly,-\H) -- cycle;
  \fill[black!10,opacity=0.6]
    (\Lx,-\Ly,0) -- (\Lx,-\Ly,-\H) -- (\Lx,\Ly,-\H) --cycle;

  \fill[black!10,opacity=0.6]
    (\Lx,-\Ly-\Lpml,0) -- (\Lx,-\Ly-\Lpml,-\H) -- (\Lx,-\Ly,-\H) -- (\Lx,-\Ly,0) -- cycle;

  \fill[black!20,opacity=0.6]
    (0,-\Ly,0) -- (0,-\Ly-\Lpml,0) -- (\Lx,-\Ly-\Lpml,0) -- (\Lx,-\Ly,0) -- cycle;

  %--------------------------------------------------
  % Receiver plane at z = z_r  (horizontal)
  %--------------------------------------------------
  \draw[red!60,thick,dashed]
    (0,-\Ly-\Lpml,-\zr) -- (\Lx,-\Ly-\Lpml,-\zr) -- (\Lx,\Ly+\Lpml,-\zr) -- (0,\Ly+\Lpml,-\zr) -- cycle;
  \draw[red!60,thick]
    (0,-\Ly,-\zr) -- (\Lx,-\Ly,-\zr) -- (\Lx,\Ly,-\zr) -- (0,\Ly,-\zr) -- cycle;
  %--------------------------------------------------
  % Water fill (semi-transparent)
  %--------------------------------------------------
  % top surface
  \fill[blue!20,opacity=0.4]
    (0,-\Ly,0) -- (\Lx,-\Ly,0) -- (\Lx,\Ly+\Lpml,0) -- (0,\Ly+\Lpml,0) -- cycle;
  % front side (y=0)
  %\fill[blue!10,opacity=0.4]
  %  (0,-\Ly-\Lpml,0) -- (\Lx,-\Ly-\Lpml,0) -- (\Lx,-\Ly-\Lpml,-\H) -- (0,-\Ly-\Lpml,-\H) -- cycle;
  % left side (x=0)
  \fill[blue!10,opacity=0.4]
    (0,-\Ly-\Lpml,0) -- (0,\Ly+\Lpml,0) -- (0,\Ly+\Lpml,-\H) -- (0,-\Ly-\Lpml,-\H) -- cycle;

  \fill[blue!10,opacity=0.4]
    (0,\Ly,-\H) -- (\Lx,\Ly,-\H) -- (\Lx,\Ly+\Lpml,-\H) -- (0,\Ly+\Lpml,-\H) -- cycle;

  \draw[thick]
    (0,-\Ly,0) -- (\Lx,-\Ly,0) -- (\Lx,\Ly,-\H) -- (0,\Ly,-\H) -- cycle;

  % top rectangle + PML (water surface) at z = 0
  \draw[dashed]
    (0,-\Ly-\Lpml,0) -- (\Lx,-\Ly-\Lpml,0) -- (\Lx,\Ly+\Lpml,0) -- (0,\Ly+\Lpml,0) -- cycle;

  % bottom rectangle + PML at z = H  (drawn as z = -H)
  \draw[dashed]
    (0,-\Ly-\Lpml,-\H) -- (\Lx,-\Ly-\Lpml,-\H) -- (\Lx,\Ly+\Lpml,-\H) -- (0,\Ly+\Lpml,-\H) -- cycle;

  % top rectangle (water surface) at z = 0
  \draw[thick]
    (0,-\Ly,0) -- (\Lx,-\Ly,0) -- (\Lx,\Ly,0) -- (0,\Ly,0) -- cycle;

  % bottom rectangle at z = H  (drawn as z = -H)
  \draw[thick]
    (0,-\Ly,-\H) -- (\Lx,-\Ly,-\H) -- (\Lx,\Ly,-\H) -- (0,\Ly,-\H) -- cycle;

  % vertical edges
  \draw[thick] (0,-\Ly,0) -- (0, -\Ly, -\H);
  \draw[thick] (\Lx,-\Ly,0) -- (\Lx, -\Ly, -\H);
  \draw[thick] (0,\Ly,0) -- (0, \Ly, -\H);
  \draw[thick] (\Lx,\Ly,0) -- (\Lx, \Ly, -\H);

  % vertical edges+PML
  \draw[dashed] (0,-\Ly-\Lpml,0)      -- (0,-\Ly-\Lpml,-\H);
  \draw[dashed] (\Lx,-\Ly-\Lpml,0)    -- (\Lx,-\Ly-\Lpml,-\H);
  \draw[dashed] (\Lx,\Ly+\Lpml,0)  -- (\Lx,\Ly+\Lpml,-\H);
  \draw[dashed] (0,\Ly+\Lpml,0)    -- (0,\Ly+\Lpml,-\H);

  %--------------------------------------------------
  % Coordinate axes (origin at back-left, surface corner)
  %--------------------------------------------------
  \draw[->,thick] (0,0,0) -- (\Lx+0.8,0,0) node[labelbox,anchor=east] {\Large$x$};
  \draw[->,thick] (0,\Ly+\Lpml+0.8,0) -- (0,-\Ly-\Lpml-0.8,0) node[labelbox,anchor=south] {\Large$y$};
  % z-axis pointing downward (depth)
  \draw[->,thick] (0,0,0) -- (0,0,-\H-0.6) node[labelbox,anchor=east] {\Large$z$};

  %---------------------------------------------------
  % Tick labels
  %---------------------------------------------------
  \draw[thick] (0,-\Ly,0) -- ++(-\ticklen,0,0) node[labelbox,anchor=west] {\Large $-L_y^*$ };
  \draw[thick] (0,\Ly,0) -- ++(-\ticklen,0,0) node[labelbox,anchor=west] {\Large $L_y^*$ };
  \draw[thick] (0,-\Ly-\Lpml,0) -- ++(-\ticklen,0,0) node[labelbox,anchor=west] {\Large $-L_y$ };
  \draw[thick] (0,\Ly+\Lpml,0) -- ++(-\ticklen,0,0) node[labelbox,anchor=west] {\Large $L_y$ };

  \draw[thick] (0,0,0) -- ++(-\ticklen,0,0) node[labelbox,anchor=west] {\Large $(0,0,0)$ };

  \draw[thick] (0,0,-\H) -- ++(-\ticklen,0,0) node[labelbox,anchor=west] {\Large $z=H$ };

  \draw[red,thick] (0,0,-\zr) -- ++(-\ticklen,0,0) node[labelbox,anchor=west] {\Large $z=z_r$};

  % labels for z=0 (surface) and z=H (bottom)

  % \draw[->,thick] (\Lx/4,-\Ly/4*3,0) -- (\Lx/2, -\Ly/4*3,0) node[labelbox,anchor=north,align=left] {propagation\\direction};

  %--------------------------------------------------
  % Acoustic point source at (x,y,z) = (0,0,z_s)
  %--------------------------------------------------
  \draw[fill=black] (0,0,-\zs) circle (0.04);
  \node[labelbox,anchor=east,align=left] at (0,0,-\zs)
    {\Large source\\$z=z_s$};
  \node[above,font=\Large] at (current bounding box.north) {(b)};

\end{tikzpicture}
        \end{adjustbox}
     \end{subfigure}
    \caption{The domains of the mode parabolic equation in underwater acoustics. 
    The dashed box contains the computational domain $\Omega_{\textrm{comp}}=(0,R)\times(-L_y^*, L_y^*)\times (0,H)$  consisting of the physical domain $\Omega_{\textrm{phys}}=(0,R)\times(-L_y^*, L_y^*)\times (0,H)$ represented by the continuously drawn box and a perfectly matched absorbing layer in $y$. 
    An initial impulse emitted by the source at $(0,0,z_s)$ is propagated in $x$ direction. 
    The results are presented at the plane $z=z_r$, shown in red.
    (a) displays the general setting, while (b) shows the specific wedge test case, with the bottom slope across the propagation direction $x$.}
    \label{fig:flat_bottom_schematic}
\end{figure}

\begin{comment}
\begin{figure}
    \centering
    \input{images/tikz_plots/ocean_wedge}
    \caption{An example from underwater acoustics. 
    The difference to Fig.~\ref{fig:flat_bottom_schematic} is that the bottom is now a wedge (i.e.\ a shoreline) and the field is propagated along it.}
    \label{fig:wedge_schematic}
\end{figure}
\end{comment}
It has been shown recently that SSP method is particularly efficient for solving Eq.~\eqref{eq:ampl_evolution} \cite{PETROV2020115526, petrovantoine2020} as was already mentioned in Sec.~\ref{sec:SSP_classical}. 
In this section we that its spectral counterpart is also excels in this task.

% \begin{figure}[ht!]
%     \centering
%     \begin{subfigure}{\linewidth}
%         \include{images/field_radio_FDM}
%     \end{subfigure}
%     \begin{subfigure}{\linewidth}
%         \include{images/field_radio_SSSP}
%     \end{subfigure}
%     \begin{subfigure}{\linewidth}
%         \include{images/radio_comparison}
%     \end{subfigure}
%     \label{fig:radio}
%     \caption{Propagation of a radio wave with $\lambda=1$\,m through a discontinuous sound speed profile. The discontinuity at $z=100$\,m produces reflections in the field. Both methods were run with the $[1/1]$ Pad\'e expansion. The first plot displays a reference field computed by a classical finite difference (FD) based SSP method, the second one is produced by the SSSP.
%     The third plot compares a slice at $y=30$\,m between the}
% \end{figure}

%%%%%%%%%%%%%%%%%%%%%
\subsection{Test scenarios}

In the framework of the mode parabolic equations theory, we consider two scenarios of sound propagation in shallow water.

First, we examine a Pekeris-type waveguide with a flat bottom to verify the capabilities of the PML (see Fig.~\ref{fig:flat_bottom_schematic}a).
Then, we address the classic issue of sound propagation in a coastal wedge \cite{jensen2011computational}, i.e., a shallow sea with a sloping bottom (see Fig.~\ref{fig:flat_bottom_schematic}b). 
In both scenarios, our waveguide consists of an upper layer representing seawater (with a sound speed of $c=1500$\,m/s and a density $\rho=1$\,g/cm$^3$) and a bottom sediment (sound speed $c=1700$\,m/s, density $\rho=1.5$\,g/cm$^3$).
The computational domain is chosen as $\Omega_{\mathrm{phys}} = \{(x,y,z)|0 \le x \le 25\,\mathrm{km}$, $-4\,\mathrm{km} \le y \le 4\,\mathrm{km}$, $0 \le z\le400\,\mathrm{m}$\}. 
The acoustic field $p(x,y,z)$ is computed using Eq.~\eqref{eq:modeansatz}, and the normal modes $\phi_j(z,x,y)$ and their wavenumbers $k_j(x,y)$ are obtained using the \texttt{ac\_modes} solver \cite{PETROV2020115526}. 
The proposed SSSP method is then used to solve equations \eqref{eq:ampl_evolution} with initial conditions $\mathcal{A}_j(0,y)$, which are obtained using the so-called ray starter \cite{PETROV2020115526}. $400$\,m-wide PMLs are added on both sides of the domain (at $y=\pm L^*_y$).

In both cases, we compute the acoustic field due to a point source of the frequency $f=25$\,Hz, located at $z_s=100$\,m, $y_s=0$\,m, $x_s=0$\,m. 
We compare and visualize the computational results for the horizontal plane $z=z_r=30$\,m. 
A $[9/9]$-Pad\'e  approximation is used in both cases, and  $N_m = 8$ modes are taken into account.

To solve the linear system for all numerical simulations, the BiCGStab method used a multiplicative preconditioner $\cX_0^{-1} = (1 + \delta k^2)^{-1}(1 + \partial_y^2)^{-1}$.
The method was run up to a relative tolerance of $10^{-8}$, though higher tolerances yield accurate results, both qualitatively and quantitatively. 
%
%\daniel{CITATION NEEDED!, It's barely mentioned in Bunker/Jensen and not mentioned at all in Collins book. I just took it from Pavels SSP code.}).

% Using an iterative method to obtain a field at the next range step introduces method parameters, specifically those of GMRES. 
% Modifying the restart period, maximum iteration count, and absolute and relative tolerances can significantly speed up the method. 
% However, this comes at the cost of precision in the next range step, as the tolerances and maximum iteration count are modified.

%%%%%%%%%%%%%%%%%%%%%%%%%%%%
\subsubsection{Flat Bottom}
In the first numerical experiment acoustic field is computed for a shallow-water Pekeris waveguide, 
displayed with the constant water depth of $h_b=200$\,m (i.e., with flat horizontal bottom, see Fig.~\ref{fig:flat_bottom_schematic}a). 
The domain was discretized into 512 points in $y$, corresponding to the mesh size $17.187$\,m, 
or a bit less than 4 points per wavelength and into $500$ points in $x$, corresponding to the step size $h = 50$\,m.

The resulting acoustic field at exhibiting a typical circular interference pattern is displayed in Fig.~\ref{fig:ocean_flat} 
(we do not present reference solution, but it ideally coincides with the one obtained using SSSP). 
Note that the utilized discretization step in $y$ is coarser than a finite difference method permits \cite{petrovantoine2020}.

\begin{figure}[htb]
    \centering
    \includegraphics[width=\linewidth]{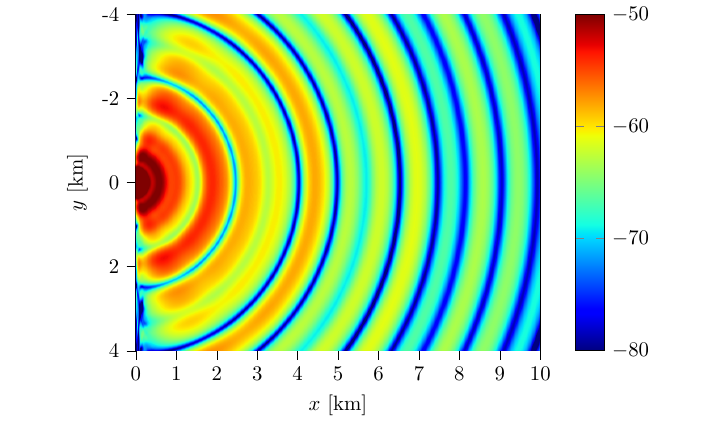}
    \caption{Acoustic pressure in dB re 1\,m at $z_r=30$\,m in the Pekeris wave\-guide.}
    \label{fig:ocean_flat}
\end{figure}

%%%%%%%%%%%%%%%%%%%%%%%%%%%%%%%%%%
\subsubsection{Wedge}
In the coastal wedge benchmark problem \cite{jensen2011computational} the water depth is a function of $y$: 
\begin{equation*}
  h_b = h_b(y) = h_{b,0} + \tan(\alpha) y\,,
\end{equation*}
where $h_{b,0} = 200$\,m, and the bottom slope angle is $\alpha \approx 2.86^\circ$.

A step size of $50$\,m was used again, and to demonstrate numerical convergence of the discretization in $y$, 
the resulting fields are displayed for 512 and 4096 $y$ points, corresponding to discretization constants of $17.1875$\,m and $\sim2.15$\,m, respectively. 
%\textcolor{red}{Daniel, is it what we see in Fig.~5a and Fig.~5b? Please cite the figure explicitly in the text. Also please state number of points per wavelength!}

The acoustic field in the wedge exhibits a remarkably different interference pattern due to the horizontal refraction effect, which causes acoustic energy to propagate towards the deeper part of the sea area. Consequently, the field is no longer symmetric with respect to the $x$-axis, as shown in Fig.~\ref{fig:ocean_wedge}.
In particular, the water-borne modes experience a cutoff around $z\approx30$, thus producing a silent zone near the coast.
The displayed fields agree qualitatively and quantitatively with the reference solution by the source image method \cite{deane1993analysis}, as shown in Fig.~\ref{fig:ocean_wedge}c, even when a coarse grid is used.

Despite not considering mode interactions, the SSSP method produces a highly accurate field with a small number of discretization points.
%\textcolor{red}{Daniel, can we check if my solution from \cite{PETROV2020115526} with $h=500$\,m and $h=1000$\,m can be also reproduced using 512 grid points in $y$? That would be spectacular!}

\begin{figure}[htb]
    \centering
    \includegraphics[width=\linewidth]{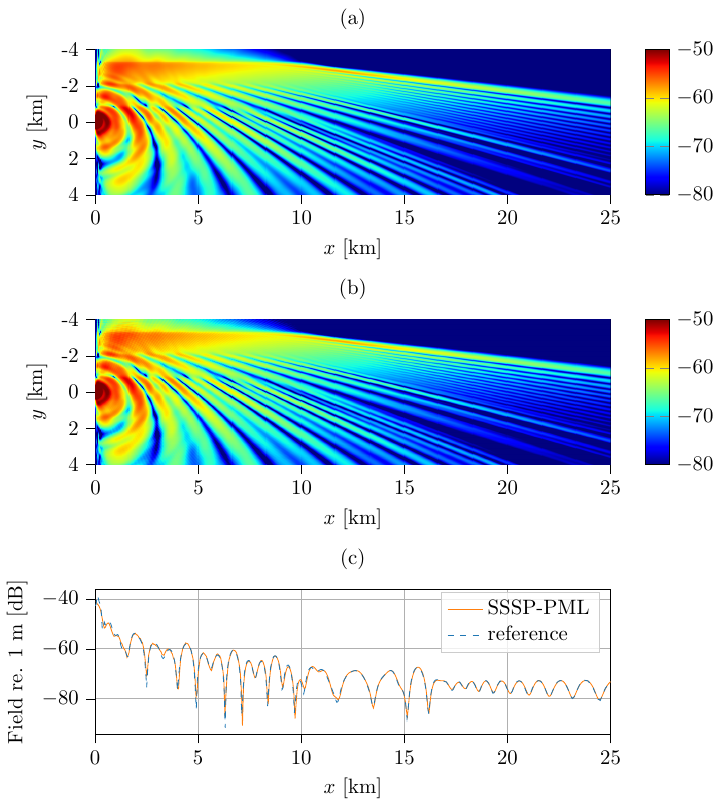}
    \caption{Acoustic field (in dB re 1 m) in the wedge at $z_r=30$\,m. (a) using 4096 points along $y$, (b) only using 512 points along $y$ and (c) comparing the field (b) with the method of source images.}
    % \caption{The acoustic field inside the wedge, measured at $z_r=50$\,m. 
    % (a) shows the field computed using 4096 points along $y$, 
    % (b) uses 512 points. This corresponds to a discretization constant of $15$\,m along $y$ for (b),
    % at a wavelength in water of $60$\,m, i.e.\ 4 points per wavelength, illustrating that the spectral method still yields physically correct results on very coarse grids. 
    % To compare the solution to a well-known method in (c), the field along the $x=0$, $z=z_r$ line was computed using the method of source images \cite{, tang2018method} as a reference.}
    \label{fig:ocean_wedge}
\end{figure}
% \begin{figure}
%     \centering
%     \begin{subfigure}{\linewidth}
%         \include{images/field_wedge_4096}
%         \label{fig:wedge_4096}
%     \end{subfigure}
%     \begin{subfigure}{\linewidth}
%         \include{images/wedge_512}
%         \label{fig:wedge_512}
%     \end{subfigure}
%     \begin{subfigure}{\linewidth}
%         \include{images/cross_wedge_512}
%     \end{subfigure}
%     \caption{The acoustic field inside the wedge, measured at $z_r=50$\,m. 
%     (a) shows the field computed using 4096 points along $y$, 
%     (b) uses 512. This corresponds to a discretization constant of $15$\,m along $y$ for (b), at a wavelength in water of $60$\,m, i.e.\ 4 points per wavelength. (b) illustrates how the spectral method still yields physically correct results on very coarse grids. To compare the solution to a well known method in (c), the field along the $x=0$, $z=z_r$ line was computed using the method of source images \cite{deane1993analysis, tang2018method} as a reference.}
%     \label{fig:wedge4096}
% \end{figure}

\clearpage

%%%%%%%%%%%%%%%%%%%%%%%%%%%%%%%%%%%%%%%%%%%%%
\section{Conclusion}\label{sec:5}

This study presents a spectral SSP algorithm for wave propagation modeling problems in the frequency domain. 
This marching numerical scheme integrates one-way counterparts of the Helmholtz equation, 
though it can be used for other elliptic problems, such as Lam\'e equations in elasticity theory or a coupled system of equations for normal mode amplitudes. The novelty of the proposed method consists in the fact that it uses a spectral representation (via FFT) to compute derivatives, whereas standard SSP-based solvers use finite differences for this purpose. 
In the presence of perfectly matched layers, the inversion of the transverse differential operator is accomplished using the iterative Bi-CGSTAB method. Although similar approaches have been proposed in the literature for certain classes of PDEs \cite{antoine2020pseudospectral,antoine2022pseudospectral}, SSSP method proposed here is the first one to deal with a pseudodifferential equation \eqref{eq:propagation}. The capabilities of the developed approach are demonstrated through two examples of realistic propagation scenarios from underwater acoustics and radiowave theory (see the GitLab repository \url{https://git.uni-wuppertal.de/walsken/sssp-pml}).

Theoretically, evaluating one Pad\'e term $w_j$, is of order $O(N)$. 
For example, the second-order finite-difference approximation of the transverse operator $\mathcal{X}$ requires $O(N\log(N))$ operations, 
whereas evaluating one Pad\'e term requires only $O(N)$. 
Thus, one might question the pursuit of making the SSP method spectral. 
However, this superficial complexity estimate applies only to traditional multi-core or single-core architectures based on serial CPUs. 
In the context of highly parallel computing systems, including modern GPUs and FPGAs, FFT libraries that are optimized to use massive concurrent threads can easily outperform traditional sequential algorithms for sparse matrix inversion. 
At the same time, the spectral accuracy of the derivative approximation within the SSSP method allows one to achieve lower computational errors with fewer discretization points, as shown in \cite{Daniel1}, 
especially in cases of continuous variation of the media parameters. 
This is particularly evident in mode parabolic equations in underwater acoustics, 
where the ocean environment typically exhibits slow and smooth variations in the $x$ and $y$ dimensions. 
The situation differs when a wide-angle parabolic equation is solved in the vertical plane (i.e., in the $x$ and $z$ coordinates). 
Although media parameters are expected to vary slowly in deep-water propagation scenarios, finite differences are much more practical in shallow-water geoacoustic waveguides with multiple interfaces where sound velocity has discontinuities. 
Thus, the presented approach could excel in specific problem classes and computational architectures rather than being a jack of all trades. 

On the other hand, we find that the question of whether SSP can be combined with spectral discretization in some way is also of theoretical interest, at least. In our opinion, this study provides a positive answer. 

%%%% Quantum computing
% 1) scheme already is an exponential propagator, which  is This is formally identical to quantum time evolution.
% 2) "Trotterization" vs. Padé approximation: these are two different approximations of the same object.
% 3) Our operator X splits naturally into: 
% differential part (diagonal in Fourier space)
% multiplication part (diagonal in real/physical space
% which is same structure as the Schrödinger equation split-step Fourier method, widely used in quantum algorithms
% 4) "Trotterization" of the square-root operator (via functional calculus / polynomial expansion)
% or approximation Approximate the generator \sqrt(1+X) by polynomial or Chebyshev expansion
% 
% The Pad\'e series of the form \eqref{eq:pade}. This is extremely interesting because Quantum algorithms also use to evaluate a summation via LCU (Linear Combination of Unitaries) methods.
% Hence, our scheme \eqref{eq:pade} is structurally a linear combination of resolvents and in quantum: linear combination of unitaries.
% BUT the PML destroy the unitarity of the evolution.

A promising direction for future work is developing quantum-compatible variants of the proposed spectral split-step Pad\'e (SSSP) framework. 
The exponential propagator underlying the method is structurally analogous to quantum time evolution and can be approximated using operator-splitting techniques, such as Trotter-Suzuki decompositions \cite{suzuki1992general} or modern Hamiltonian simulation methods \cite{lloyd1996universal}. 
In particular, the separation into Fourier-diagonal and multiplicative operators naturally aligns with quantum algorithms based on the quantum Fourier transform and the linear combination of unitaries approach \cite{berry2007efficient}. 
However, a key challenge remains: the treatment of the non-Hermitian PML operator would require extensions to open-system quantum simulation frameworks \cite{kushida2025quantum}.
% \Tmatthias{The applicability of the quantum approach depends on the application: sometimes one only wants to compute the value at the final point (receiver).}

%%%%%%%%%%%%%%%%%%%%%%%%%%%%%
%\textcolor{red}{Refs Levy book, paper Antoine PML, .....}
%\Tmatthias{cite refs below or comment them out}
%\nocite{duru2010stable}
%\nocite{antoine2020pseudospectral}
%\nocite{antoine2020perfectly}
%\nocite{antoine2022pseudospectral}
%\nocite{ehrhardt2008discrete}
%\nocite{levy2002transparent}
%\nocite{mikhin2004exact}
%% \nocite{varon2023approximation,varon2026deep}
%\nocite{porter1992kraken}
%\nocite{Hagstrom26}
%\nocite{petrov2024generalization}
%\nocite{petrov2022decomposition}
%\nocite{petrov2016wide}
%\nocite{petrov2016transparent}
%%%% BIBLIOGRAPHY %%%%
\bibliographystyle{plain}
\bibliography{references}

\appendix

\end{document}